\newcommand{\lab}[1]{\label{#1}}                

\documentclass[12pt,reqno]{amsart}
\usepackage{amsmath}
\usepackage{amssymb}
\usepackage[left=3cm,top=3cm,right=3cm,bottom=3cm]{geometry}
\usepackage{epsfig}

 \usepackage[usenames,dvipsnames]{color}
 
\numberwithin{equation}{section}

\newcommand{\remove}[1]{}

\begin{document}
\newtheorem{theorem}{Theorem}[section]
\newtheorem{lemma}[theorem]{Lemma}
\newtheorem{sublemma}[theorem]{Sub-lemma}
\newtheorem{definition}[theorem]{Definition}
\newtheorem{conjecture}[theorem]{Conjecture}
\newtheorem{proposition}[theorem]{Proposition}
\newtheorem{claim}[theorem]{Claim}
\newtheorem{algorithm}[theorem]{Algorithm}
\newtheorem{corollary}[theorem]{Corollary}
\newtheorem{observation}[theorem]{Observation}
\newtheorem{problem}[theorem]{Open Problem}
\newcommand{\R}{{\mathbb R}}
\newcommand{\N}{{\mathbb N}}
\newcommand{\Z}{{\mathbb Z}}
\newcommand\eps{\varepsilon}
\newcommand{\E}{\mathbb E}
\newcommand{\Prob}{\mathbb{P}}
\newcommand{\pl}{\textrm{C}}
\newcommand{\dang}{\textrm{dang}}
\renewcommand{\labelenumi}{(\roman{enumi})}
\newcommand{\bc}{\bar c}
\newcommand{\cal}[1]{\mathcal{#1}}
\newcommand{\G}{{\cal G}}
\newcommand{\Hc}{{\cal H}}
\newcommand{\Gnd}{\G_{n,d}}
\newcommand{\Gnp}{\G(n,p)}
\renewcommand{\P}{{\cal P}}
\newcommand{\la}{\lambda}
\newcommand{\floor}[1]{\lfloor #1 \rfloor}

\newcommand{\bel}[1]{\be\lab{#1}}
\newcommand{\ee}{\end{equation}}
\newcommand{\be}{\begin{equation}}
 \newcommand\eqn[1]{(\ref{#1})}
 \newcommand{\ex}{\E}

\newcommand{\aas}{{a.a.s.}}
\newcommand{\wO}{\widetilde O}
\newcommand{\accessconst}{\gammaconst}
\newcommand{\gammaconst}{9}
\newcommand{\oldiii}{(iii)[[[*** this will be (iv) ***]]]}
\newcommand{\newiii}{[[[*** New part (iii) ***]]]}

\newcommand{\Aconst}{a} 
\newcommand{\Bconst}{b} 
\newcommand{\hatU}{\widehat U}
\newcommand{\Bin}{{\rm Bin}}
\newcommand{\tildeU}{{\widetilde U}}

\title{Meyniel's conjecture holds for random graphs} 

\author{Pawe\l{} Pra\l{}at}
\address{Department of Mathematics, Ryerson University, Toronto, ON, Canada, M5B 2K3}
\thanks{The first author was supported by NSERC and Ryerson University}
\email{\texttt{pralat@ryerson.ca}}

\author{Nicholas Wormald}
\thanks{The second author was supported by the Canada Research Chairs Program and NSERC, and in part by  Australian Laureate Fellowships grant FL120100125.}
\address{School of Mathematical Sciences, Monash University  VIC 3800, Australia}
\email{\tt nick.wormald@monash.edu}


\maketitle

\begin{abstract}
In the game of cops and robber, the cops try to capture a robber moving on the vertices of the graph. The minimum number of cops required to win on a given graph $G$ is called the cop number of $G$. The biggest open conjecture in this area is the one of Meyniel, which asserts that for some absolute constant $C$,  the cop number of every connected graph $G$ is at most $C \sqrt{|V(G)|}$. In this paper, we show that Meyniel's conjecture holds asymptotically almost surely for the binomial random graph $\Gnp$, which improves upon existing results showing that asymptotically almost surely the cop number of $\Gnp$ is $O(\sqrt{n} \log n)$  provided that $pn \ge (2+\eps) \log n$ for some $\eps > 0$. We do this by first showing that the conjecture holds for a general class of graphs with some specific expansion-type properties. This will also be used in a separate paper on random $d$-regular graphs, where we show that the conjecture holds asymptotically almost surely when  $d = d(n) \ge 3$.
\end{abstract}

\section{Introduction\label{intro}}

The game of \emph{Cops and Robbers}, introduced independently by Nowa\-kowski  and Winkler~\cite{nw} and Quilliot~\cite{q} almost thirty years ago, is played on a fixed graph $G$. We will always assume that $G$ is undirected, simple, and finite.  There are two players, a set of $k$ \emph{cops}, where $k\ge 1$ is a fixed integer, and the \emph{robber}.  The cops begin the game by occupying any set of $k$ vertices (in fact, for a connected $G$, their initial position does not matter). The robber then chooses a vertex, and the cops and robber move in alternate turns. The players use edges to move from vertex to vertex.  More than one cop is allowed to occupy a vertex, and the players may remain on their current positions. The players know each others current locations.  The cops win and the game ends if at least one of the cops eventually occupies the same vertex as the robber;  otherwise, that is, if the robber can avoid this indefinitely, he wins. As placing a cop on each vertex guarantees that the cops win, we may define the \emph{cop number}, written $c(G)$, which is the minimum number of cops needed to win on $G$. The cop number was introduced by Aigner and Fromme~\cite{af} who proved (among other things) that if $G$ is planar, then $c(G)\leq 3$. For more results on vertex pursuit games such as \emph{Cops and Robbers}, the reader is directed to the  surveys on the subject~\cite{al,ft,h} and the recent monograph~\cite{bn}.

The most important  open problem in this area is Meyniel's conjecture (communicated by Frankl~\cite{f}). It states that $c(n) = O(\sqrt{n})$,  where $c(n)$ is the maximum of $c(G)$ over all $n$-vertex connected graphs.  If true, the estimate is best possible as one can construct a bipartite graph based on the finite projective plane with the cop number of order at least  $\sqrt{n}$. Up until recently, the best known upper bound of $O(n \log \log n / \log n)$ was given in~\cite{f}. It took 20 years to show that $c(n) = O(n/\log n)$ as proved in~\cite{eshan}. Today we know that the cop number is at most $n 2^{-(1+o(1))\sqrt{\log_2 n}}$ (which is still $n^{1-o(1)}$) for any connected graph on $n$ vertices (the result obtained independently by Lu and Peng~\cite{lp}, Scott and Sudakov~\cite{ss}, and Frieze, Krivelevich and Loh~\cite{fkl}).  If one looks for counterexamples for Meyniel's conjecture it is natural to study first the cop number of random graphs. However, this paper shows that Meyniel's conjecture holds asymptotically almost surely for random graphs.

Let us recall two classic models of random graphs that we study in this paper. The binomial random graph $\Gnp$ is defined as a random graph with vertex set $[n]=\{1,2,\dots, n\}$ in which a pair of vertices appears as an edge with probability $p$, independently for each such a pair. As typical in random graph theory, we shall consider only asymptotic properties of $\Gnp$ as $n\rightarrow \infty$, where $p=p(n)$ may and usually does depend on $n$. Another probability space is the one of random $d$-regular graphs on $n$ vertices with uniform probability distribution. This space is denoted $\mathcal{G}_{n,d}$, with $d\ge 2$ fixed, and $n$ even if $d$ is odd. We say that an event in a probability space holds \emph{asymptotically almost surely} (\emph{a.a.s.}) if its probability tends to one as $n$ goes to infinity.

Let us first briefly describe some known results on the cop number of $\Gnp$. Bonato, Wang, and the first author investigated such games in $\G(n,p)$ random graphs, and their generalizations used to model complex networks with a power-law degree distribution (see~\cite{bpw, bpw2}). From their results it follows that if $2 \log n / \sqrt{n} \le p < 1-\eps$ for some $\eps>0$, then a.a.s. 
$$
c(\Gnp)= \Theta(\log n/p),
$$
so Meyniel's conjecture holds a.a.s.\ for such $p$. A simple argument using dominating sets shows that Meyniel's conjecture also holds a.a.s.\  if $p$ tends to 1 as $n$ goes to infinity (see~\cite{p} for this and stronger results). Recently, Bollob\'as, Kun and Leader~\cite{bkl} showed that for $p(n) \ge 2.1 \log n /n$, then a.a.s.
$$
\frac{1}{(pn)^2}n^{ 1/2 - 9/(2\log\log (pn))  }  \le c(\Gnp)\le 160000\sqrt n \log n\,.
$$
From these results, if $np \ge 2.1 \log n$ and either $np=n^{o(1)}$ or $np=n^{1/2+o(1)}$, then a.a.s.\ $c(\Gnp)= n^{1/2+o(1)}$. Somewhat surprisingly, between these values $c(\Gnp)$ was shown by \L{}uczak and the first author~\cite{lp2}   to have more complicated behaviour.
\begin{theorem}[\cite{lp2}]\lab{thm:zz}
Let $0<\alpha<1$, let $j\ge 1$ be integer,  and let $d=d(n)=(n-1)p=n^{\alpha+o(1)}$.
\begin{enumerate}
\item If $\frac{1}{2j+1}<\alpha<\frac{1}{2j}$, then a.a.s.\
$$
c(\Gnp)= \Theta(d^j)\,.
$$
\item If $\frac{1}{2j}<\alpha<\frac{1}{2j-1}$, then a.a.s.\
$$
 \frac{n}{d^j} =O\big( c(\Gnp)\big)= O \left( \frac{n}{d^j} \log n \right)\,.
$$
\item If $\alpha = 1/(2j)$ or $\alpha =1/(2j+1)$, then a.a.s.\ $c(\Gnp) < d^{j+o(1)}$.
\end{enumerate}
\end{theorem}
It follows that a.a.s.\ $\log_n  c(\G(n,n^{x-1}))$ is asymptotically bounded above by   the function $f(x)$ shown in Figure~\ref{fig1}.
\begin{figure}
\begin{center}
\includegraphics[width=4in]{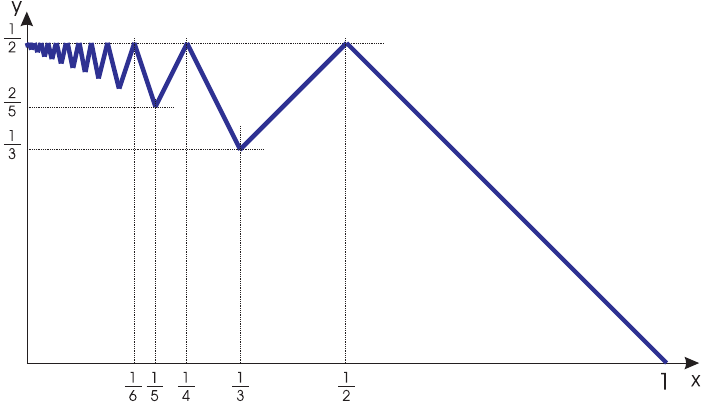}
\end{center}
\caption{The ``zigzag'' function $f$.}\label{fig1}
\end{figure}
From the above results, we know that Meyniel's conjecture holds a.a.s.\ for random graphs except perhaps when $np=n^{1/(2k)+o(1)}$ for some $k \in \N$, or $np=n^{o(1)}$. We show in this paper that the conjecture holds \aas\ in $\Gnp$ provided that $np> (1/2+\eps) \log n$ for some $\eps > 0$.  

\begin{theorem}
Let $\eps>0$ and suppose that $d:=p(n-1)\ge (1/2+\eps) \log n$. Let $G=(V,E) \in \Gnp$. Then a.a.s.\
$$
c(G) = O(\sqrt{n}).
$$
\end{theorem}

Note that Meyniel's conjecture is restricted to connected graphs, but $G \in \G(n,p)$ is \aas\ disconnected when $np\le (1-\eps) \log n$. Thus, we have shown that the following version of Meyniel's conjecture holds a.a.s.\ for $G \in \G(n,p)$ for {\em all} $p$: if $G$ is connected then $c(G)=O(\sqrt n)$. This of course implies the corresponding result for the $\G(n,m)$ model of random graphs. These results for random graph models support Meyniel's conjecture although there is currently a huge gap in the deterministic bounds: it is still not known whether there exists $\eps>0$ such that the cop number of connected graphs is $O(n^{1-\eps})$.  

We consider dense graphs in Section~\ref{s:dense} and sparse graphs in Sections~\ref{s:sparse-det} and~\ref{s:sparse}. 
In each case we first show that the conjecture holds deterministically for a general class  of graphs  with some specific expansion-type properties. We then show that $G\in\Gnp$ is \aas\   contained in the general class. The deterministic result is more complicated in the sparse case so is treated separately in Section~\ref{s:sparse-det}. These deterministic results will also be used in a separate paper on random $d$-regular graphs~\cite{PW_d-reg}, where we show that the conjecture holds \aas\ for $G \in \mathcal{G}_{n,d}$ when $d = d(n) \ge 3$. The main result is a combination of Theorems~\ref{thm:dense_case} and~\ref{t:lowdreggnp}.

\section{Preliminaries}

Before stating the result, we need some definitions. Let $S(v,r)$ denote the set of vertices whose distance from $v$ is precisely $r$, and $N(v,r)$ the set of vertices (``ball'') whose distance from $v$ is at most $r$. Also, $N[S]$ denotes $\bigcup_{v \in S} N(v,1)$, the closed neighbourhood of $S$,  and $N(S)=N[S] \setminus S$ denotes the (open) neighbourhood of $S$. Finally, we define $S(U,r)$ to be the set of vertices whose distance to $U$ is exactly $r$, and  $N(U,r)$ the set of vertices of distance at most $r$ from $U$. All logarithms with no suffix are natural. We use $\wO(f(n))$ to denote  $O(f(n))\log^{O(1)} n$.

Throughout the paper, we will be using the following concentration inequalities. Let $X \in \textrm{Bin}(n,p)$ be a random variable with the binomial distribution with parameters $n$ and $p$. Then, a consequence of Chernoff's bound (see e.g.~\cite[Corollary~2.3]{JLR}) is that 
\bel{chern}
\Prob( |X-\E X| \ge \eps \E X) ) \le 2\exp \left( - \frac {\eps^2 \E X}{3} \right)  
\ee
for  $0 < \eps < 3/2$. This inequality will be used many times but at some point we will also apply the following, more common, form of Chernoff's bound:  
\bel{chern2}
\Prob(|X-np|>a)<2e^{-2a^2/n},
\ee 
and a more basic version (see e.g.~\cite[Theorem 2.1]{JLR}):
\bel{chern-strong}
\Prob( X \le \E X-t ) \le \exp \big( - \E X \psi(-t/\E X) \big),
\ee
where $\psi(x)=(1+x) \log (1+x) - x$, $x > -1$. 

Finally, we will also use the bound of Bernstein (see e.g.~\cite[Theorem~2.1]{JLR}) that 
\begin{eqnarray}\label{Bernstein}
\Prob \left( X \ge (1+x) \E X \right) &\le&  \exp \left( - \frac {x^2  \E X} {2(1+x/3)} \right).
\end{eqnarray}

\section{Dense case}\lab{s:dense}

In this section, we focus on dense random graphs, that is, graphs with average degree $d = p(n-1) \ge \log^3 n$. We will prove a general purpose result that holds for a family of graphs with some specific expansion properties. After that we will show that dense random graphs \aas\ fall into this class of graphs and so the conjecture holds \aas\ for dense random graphs. 

\begin{theorem}\label{thm:general_dense_case}
Let $\G_n$ be a set of graphs and $d=d(n)\ge \log^3 n$. Suppose that for some positive constant $c$, for all $G_n\in\G_n$ the following properties hold.
\begin{enumerate}
\item Let $S \subseteq V(G_n)$ be any set of $s=|S|$ vertices, and let $r \in \N$. Then
$$
\left| \bigcup_{v \in S} N(v,r) \right| \ge c \min\{s d^r, n \}.
$$
Moreover, if $s$ and $r$ are such that $s d^r < n / \log n$, then
$$
\left| \bigcup_{v \in S} N(v,r) \right| \sim s d^r.
$$
\item Let $v \in V(G_n)$, and let $r \in \N$ be such that $\sqrt{n} < d^{r+1} \le \sqrt{n} \log n$. Then there exists a family 
$$
\Big\{W(u) \subseteq S(u,r+1) : u \in S(v,r) \Big\}
$$ 
of pairwise disjoint subsets such that, for each $ u \in S(v,r)$,
$$
|W(u)| \sim d^{r+1}.
$$
\end{enumerate}
Then $c(G_n) = O(\sqrt{n}).$
\end{theorem}

Before we move to the proof of Theorem~\ref{thm:general_dense_case}, we need an observation.  

\begin{lemma}\lab{lem:assignment}
Suppose that $d=p(n-1) \ge \log^3 n$. Let $G=(V,E)$ be a graph possessing the properties stated in Theorem~\ref{thm:general_dense_case}. Let $X \subseteq V$ be any set of at most $2 \sqrt{n}$ vertices and $r=r(n) \in \N$ is such that $d^{r} \ge \sqrt{n} \log n$. Let $Y \subseteq V$ be a random set determined by independently choosing each vertex of $v \in V$ to be in $Y$ with probability $C / \sqrt{n}$, where $C \in \R$.
Then, for sufficiently large constant $C$, the following statement holds with probability $1-o(n^{-2})$: it is possible to assign all the vertices in $X$ to distinct vertices in $Y$ such that for each $u\in X$, the vertex in $Y$ to which it is assigned is within distance $r$ of $u$. 
\end{lemma}
\begin{proof}
In order to show that the required assignment exists with probability $1-o(n^{-2})$, we show that with this probability the random choice of vertices in $Y$ satisfies the Hall condition for matchings in bipartite graphs. Set 
$$
k_0 = \max \{ k : k d^{r} < n \}.
$$ 
Let $K \subseteq X$ with $|K|=k \le k_0$. We may apply the condition in Theorem~\ref{thm:general_dense_case}(i) to bound the size of $\bigcup_{u \in K} N(u,r)$. From the definition of $k$ we have $k \sqrt{n} \log n \le  k d^{r} < n$, and hence  the number of vertices of $Y$ in $\bigcup_{u \in K} N(u,r)$ can be stochastically bounded from below by the binomial random variable  ${\rm Bin}(\lfloor c k \sqrt{n} \log n\rfloor, C/\sqrt{n})$, whose expected value is asymptotic to  $Cck \log n$.   Using Chernoff's bound~\eqn{chern} we get that the probability that there are fewer than $k$ vertices of $Y$ in this set of vertices is less than $\exp(-4k \log n)$ when $C$ is a sufficiently large constant. Hence, the probability that the sufficient condition in the statement of   Hall's theorem fails for at least one set $K$ with $|K|\le k_0$ is at most
$$
\sum_{k=1}^{k_0} {|X| \choose k} \exp( - 4 k \log n)  \le \sum_{k=1}^{k_0} n^k \exp( - 4 k \log n) = o(n^{-2}).
$$

Now consider any set  $K \subseteq X$ with $k_0 < |K| = k \le |X| \le 2 \sqrt{n}$ (if such a set exists). Note that the condition in Theorem~\ref{thm:general_dense_case}(i) implies that the size of $\bigcup_{u \in K} N(u,r)$ is at least  $cn$, so we expect at least $Cc \sqrt{n}$ vertices of $Y$ in this set.  Again using~\eqn{chern}, we deduce that the number of vertices of $Y$ in this set is at least $2 \sqrt{n} \ge |X| \ge |K|$ with probability at least $1-\exp(- 4 \sqrt{n})$, by taking the constant $C$ to be large enough. Since
$$
\sum_{k=k_0+1}^{|X|} {|X| \choose k} \exp( - 4 \sqrt{n} ) \le 2 \sqrt{n} 2^{2 \sqrt{n}} \exp( - 4 \sqrt{n} ) = o(n^{-2}),
$$
the necessary condition in Hall's theorem holds with probability $1 - o(n^{-2})$. 
\end{proof}

We now return to the proof of Theorem~\ref{thm:general_dense_case}.  The proof only relies on two things: (i)  upper bounds on the size of $N(v,r)$ where $v$ is the position of  the robber, and  (ii) lower bounds on the proportion of vertices, from an arbitrary set of vertices, that can be covered by cops in a given number of steps. As a result, a slightly stronger (and useful when dealing with random $d$-regular graphs) result holds---see Observation~\ref{obs:sandwitch} right after the proof. 

\begin{proof}[Proof of Theorem~\ref{thm:general_dense_case}]
We need to introduce two independent teams of cops that are distributed at random. (In Case 1 described below, one team is enough.) Each team of cops is determined by independently choosing each vertex of   $v \in V(G_n)$ to be occupied by such a cop with probability $C / \sqrt{n}$, where $C$ is a (large) constant to be determined soon. The total number of cops is $\Theta(\sqrt{n})$ a.a.s.

The robber appears at some vertex $v \in V(G_n)$. Let $r=r(d)$ be the smallest integer such that $d^{r+1} \ge \sqrt{n}$.  Note that it follows from $d^r< \sqrt{n}$,  and  assumption (i) that $|N(v,r)| < 2\sqrt{n}$. We consider two cases, depending on the value of $d^{r+1}$, and in each case we give a strategy   which permits the cops to win a.a.s. The first case is based on the idea used in~\cite{lp2}.

\smallskip

\noindent
\textbf{Case 1.} $d^{r+1} \ge \sqrt{n} \log n$. 

This case is rather easy. Since $|N(v,r)| < 2\sqrt{n}$ and $d^{r+1} \ge \sqrt{n} \log n$, it follows from  Lemma~\ref{lem:assignment} that with probability $1-o(n^{-1})$ it is possible to assign distinct cops from the first team to all vertices $u$ in  $N(v,r)$ such that a cop assigned to $u$  is within distance $(r+1)$ of $u$. (Note that here, the probability refers to the randomness in distributing the cops; the graph $G_n$ is fixed.) If this can be done, then after the robber appears these cops can begin moving straight to their assigned destinations  in  $N(v,r)$. Since the first move belongs to the cops, they have $(r+1)$ steps, after which the robber must still be inside $N(v,r)$, which is fully occupied by cops.

Hence, the cops will win with probability $1-o(n^{-1})$, for each possible starting vertex $v \in V(G_n)$. It follows that the strategy gives a win for the cops a.a.s.  

\smallskip

\noindent
\textbf{Case 2.}   $d^{r+1} = \sqrt{n} \cdot \omega$, where $1 \le \omega = \omega(n) < \log n$.  

Suppose that $r \ge 1$ (the case $r=0$ has to be treated differently and will be discussed at the end of the proof). This time, the first team cannot necessarily catch the robber a.a.s., since we cannot hope to find cops inside all the required neighbourhoods.  Instead, the first team of cops will be given the task of ``densely covering'' the sphere $S(v,r)$. Not every $u \in S(v,r)$ can have a distinct cop assigned. For convenience, we restrict ourselves to trying to find one in each set $W(u) \subseteq S(u,r+1)$, as defined in condition (ii). Using the estimate on  $|W(u)|$ given there, the probability that $W(u)$ contains no cop from the first team is bounded from above by
$$
\left( 1 - \frac {C}{\sqrt{n}} \right)^{|W(u)|} < \exp \left( - \frac {C}{\sqrt{n}} \cdot \frac 12 \sqrt{n} \cdot \omega \right) = \exp \left( - \frac {C}{2} \cdot \omega \right) < \frac {1}{10 \omega},
$$
for $C$ sufficiently large. 

There is a complication, that the robber might try to ``hide'' inside $N(v,r)$. We can use an auxiliary team of $C\sqrt n$ cops, randomly placed, and  for every $u\in N(v,r)$, we search for one of them within distance $(r+2)$ of $u$. Since  $|N(v,r)| < 2\sqrt{n}$ and $d^{r+2} = d \sqrt{n} \omega \ge \sqrt{n} \log n$, it follows from Lemma~\ref{lem:assignment} that these cops will catch the robber if she takes at least $(r+1)$ steps to reach $S(v,r)$. Thus, we may assume that the robber reaches $S(v,r)$ in precisely $r$ steps.

The second main team of cops is released when the robber is at $z \in S(v, \lfloor r/2 \rfloor)$. (Note that for $r=1$, we have $z=v$ and both teams start at the same time.) We may assume that after $r$ steps from the start, the robber must be located in the set $S \subseteq S(z, \lceil r/2 \rceil) \cap S(v,r)$ that is not covered by cops from the first team. Note that the robber can see both teams of cops right from the start, so she can try to use her best strategy by selecting appropriately the pair of vertices $v$ and $z$. However, since there are $O(n^2)$ possible pairs to consider, it is enough to show that, for a given starting vertex $v$, and assuming the robber appears at $z$ after $ \lfloor r/2 \rfloor$  of her  steps,  the second team of cops can then catch her in the next $(r+1)$ of her steps with probability $1-o(n^{-2})$. Note that the cops get an extra initial step when the robber appears at this place, so they can have $(r+2)$ steps available for this purpose.

From condition (i) we have $|S(z,\lceil r/2\rceil  )|\le 2d^{\lceil r/2\rceil}$. Hence, from above, the expected size of $S$ is at most
$$
2 d^{\lceil r/2 \rceil} \cdot \frac {1}{10 \omega},
$$
and $|S| \le \frac {d^{\lceil r/2 \rceil}}{4 \omega}$ with probability at least $1-o(n^{-2})$. We can therefore assume that this bound on $|S|$ holds.

Since we can assume (as noted above) that the robber will be at a vertex in $S\subseteq S(v,r)$ after taking $\lceil r/2\rceil$ steps from $z$, the second team of cops has the assignment of covering the set of vertices 
$$
U = \bigcup_{s \in S} S(s, \lfloor r/2 \rfloor + 1),
$$
of size at most 
$$
2 d^{\lfloor r/2 \rfloor + 1} |S| \le \frac {d^{r+1}}{2 \omega} < \sqrt{n},
$$
by   condition (i). Thus, for each $u \in U$, we would like to find a cop within distance $(r+2)$. Since $d^{r+2} \ge \sqrt{n} \log n$, it follows from Lemma~\ref{lem:assignment} that the second team of cops with probability at least $1-o(n^{-2})$ can occupy all of $U$ in $(r+2) $ steps after the robber reaches $z$. 

The argument above showed that the robber reached vertex $s \in S \subseteq S(v,r)$ after $r$ steps, and thus she must be on or inside one of the spheres in the definition of $U$ when the second team of cops are in place.  If she is not already caught, we can use a ``clean-up'' team (disjoint from the two main teams and the auxiliary one), of $C\sqrt n$ cops, and send them to $N(s,\lfloor r/2 \rfloor+1)$ to catch the robber while the other cops remain at their final destinations. (Note that, in fact, it follows from condition (i) that $|N(s,\lfloor r/2 \rfloor + 1)| = o(\sqrt{n})$ so an even  smaller team would do the job.) 

Now suppose that $r=0$, that is, $d \in [\sqrt{n}, \sqrt{n} \log n)$. The first team of cops is generated randomly, exactly the same way as before. However, in this case, if the first team of cops were to try to cover $N(v,1)$ in one step, since the expected number of them is   $\Theta(\sqrt{n})$,   they could normally occupy only a negligible part of the neighbourhood of $v$ if $d$ is substantially larger than $\sqrt n$. So instead, we hold the first team  at their original positions for their first move. For a sufficiently large constant $C$,   the first team a.a.s.\ dominates all of $S(v,1)$ except for a set $S\subseteq S(v,1)$ with $|S|/|S(v,1)| < e^{- C \omega/2} < 1/ (2 \omega)$. Hence, if the robber moves in the first   step, she has to go to $S$, and $|S| < \sqrt{n}$. The second team is generated the same way as the first team, but independently, and is instructed to try to cover all of $S\cup\{v\}$ in two steps. Since  $d^2 \ge n > \sqrt{n} \log n$,  Lemma~\ref{lem:assignment} implies that with probability $1-o(n^{-2})$ we can find cops, one in $S(s,2)$ for each $s \in S\cup\{v\}$, that are distinct. In this case the second team can capture the robber after her first step, on their second move. This strategy gives a win for the cops a.a.s., regardless of the initial choice for $v \in V$.
\end{proof} 

 Before we move to random graphs, let us observe that, in fact, a slightly stronger result holds. Suppose that the robber plays on a graph $G=(E_1, V)$ but the cops play on a different graph, $H=(E_2, V)$, on the same vertex set. Again, the cops win if they occupy the vertex of the robber. We will use $c(G,H)$ for the counterpart of the cop number for this variant of the game. This variant was studied before, but the main point of introducing it here is to  show that Meyniel's conjecture holds a.a.s.\ for random $d$-regular graphs, for certain $d$. In view of the note before the proof of Theorem~\ref{thm:general_dense_case}, the same conclusion holds on this variant of the game provided that the appropriate upper and lower bounds in the hypotheses hold on the respective graphs. More precisely and in particular, this implies the following.

\begin{observation}\label{obs:sandwitch}
Let $\G_n$ and $\Hc_n$ be two sets of graphs and $d=d(n)\ge \log^3 n$. Suppose that for all $G_n\in\G_n$ and all $H_n\in\Hc_n$ we have the following:
\begin{enumerate}
\item for some positive constant $c$, the conditions (i) and (ii) in the hypotheses of Theorem~\ref{thm:general_dense_case} are satisfied for $H_n$,
\item $G_n$ is $d$-regular.
\end{enumerate} 
Then $c(G_n, H_n) = O(\sqrt{n})$.
\end{observation}
It is known that if $d=o(d^{1/3})$, then there exists a coupling of $\Gnp$ with $p = (1-(\log n/d)^{1/3})$, and the space $\Gnd$ of random $d$-regular graphs, such that a.a.s.\ $\Gnp$ is a subgraph of $\Gnd$~\cite{KimVu}. Hence, since we are only concerned with $r=O(\log n/\log \log n)$, the observation implies that Meyniel's conjecture holds a.a.s.\ for random $d$-regular graphs when $d\ge \log^4 n$ and $d=o(n^{1/3})$. If suitable couplings for other ranges of $d$ are shown, they will similarly immediately imply the conjecture for such graphs.  A full treatment of the conjecture on  $\Gnd$ will appear in~\cite{PW_d-reg}.
\smallskip
  
We are now ready to  show that Meyniel's conjecture holds for dense random graphs.
\begin{theorem}\label{thm:dense_case}
Suppose that $d=p(n-1)\ge \log^3 n$. Let $G=(V,E) \in \Gnp$. Then a.a.s.\
$$
c(G) = O(\sqrt{n}).
$$
\end{theorem}

\begin{proof} 
By Theorem~\ref{thm:general_dense_case}, we just have to   show that its hypotheses, which are referred to as  (i) and (ii) in this proof,  hold a.a.s.

Let $S \subseteq V$, $s=|S|$, and consider the random variable $X = X(S) = |N(S)|$.  For (i), we will bound $X$ from below in a stochastic sense. There are two things that need to be estimated: the expected value of $X$, and the concentration of $X$ around its expectation. 

It is clear that  
\begin{eqnarray*}
\E X &=& \left( 1 - \left(1- \frac {d}{n-1} \right)^s \right) (n-s) \\
&=& \left( 1 - \exp \left( - \frac {ds}{n} (1+O(d/n)) \right) \right) (n-s) \\
&=& \frac {ds}{n} (1+O(ds/n)) (n-s) \\
&=& ds (1+O(\log^{-1} n)), 
\end{eqnarray*}
provided $ds \le n/ \log n $.  We next use Chernoff's bound (\ref{chern}) which implies that the expected number of sets $S$ that have $\big| |N(S)| - d|S| \big| > \eps d|S|$ and $|S| \le n/(d\log n)$  is, for $\eps = 2/{\log n}$, at most 
\bel{badS}
\sum_{s \ge 1} 2n^s \exp \left( - \frac {\eps^2 s \log^3 n}{3+o(1)} \right) 
= o(1).
\ee
So a.a.s.\ if $ |S| \le n/d\log n\mbox{ then } |N(S)| =d |S| (1+O(1/\log n))$ where the constant implicit in $O()$  does not depend on the choice of $S$.  Since $d \ge \log^3 n$, for such sets we have $|N[S]| = |N(S)| (1+O(1/d)) = d |S| (1+O(1/\log n))$. We may assume  this statement holds. 
 
Given this assumption, we have good bounds on the ratios of the cardinalities of $N[S]$, $N[N[S]] = \bigcup_{v \in S} N(v,2)$,  and so on. We consider this up to the $r$'th iterated neighbourhood  provided    $sd^r \le  {n}/{\log n}$ and thus  $r = O(\log n /\log \log n)$. Then the cumulative multiplicative error term  is $(1+O(\log^{-1}n))^r = (1+o(1))$, that is, 
\bel{eq:i}
\bigg| \bigcup_{v \in S} N (v,r) \bigg| \sim sd^{r} 
\end{equation} 
for all $s$ and $r$ such that $sd^r \le n / \log n$. This establishes  (i) in this case.

Suppose now that $sd^r = c'n$ with $1/ \log n< c' = c'(n) \le 1$. Using~(\ref{eq:i}), we have that $U= \bigcup_{v \in S}  N(v,r-1)$ has cardinality $(1+o(1))sd^{r-1} $. Now $N[U]$ has expected size  
$$
n - e^{-c'}n(1+o(1)) \ge \frac 12 c'n(1+o(1)) = \frac 12 sd^r(1+o(1)),
$$
since $c' \le 1$. Chernoff's bound~\eqn{chern} can be used again in the same way as before to show that with high probability  $|N[U]|$ is concentrated near its expected value, and hence that  a.a.s.\ $|N[U]|> \frac{4}{9} sd^r$ for all sets $S$ and all $r$ for which  $n/\log n < sd^r \le n$, where $s=|S|$. Thus (i) holds also in this case.

Finally, if $sd^r > n$, consider the maximum $r_0$ such that $sd^{r_0} \le n$.  From the penultimate conclusions of the previous two cases, it follows that for  $U=\bigcup_{v \in S}  N(v,r_0)$ we have $|U|\ge \frac{4}{9} sd^{r_0} \ge \frac{4}{9}n/d $.   Next  we show that taking one more step, i.e.\ at most $r$ steps in total, will reach at least $cn$  vertices for some universal constant $c>0$, as required.  Applying~(\ref{chern2}) and taking the union bound over all $S$ with $|S|=s$, we     deduce that for any $\eps>0$, with probability conservatively at least $1- o(n^{-3})$ we have
$$
\left|\bigcup_{v \in S} N(v,r_0+1)\right|=|N[U]|> n(1-e^{-4/9}-\eps), 
$$
which is at least $n/3$ for small enough $\eps$. Of course,  for any $v \in S$,  $N(v,r_0+1)\subseteq N(v,r)$. We conclude that (i) holds a.a.s.\ for all sets $S$ and all $r$, where $s=|S|$, in this case. Hence (i) holds in full generality  a.a.s.

To prove (ii) holds a.a.s., note first that  if such an $r$ exists, then it it follows from $d^{r+1} \le \sqrt{n} \log n$ and $d \ge \log^3 n$ that $d^r \le \sqrt{n} / \log^2 n$, and in particular $r$ is uniquely determined from $d$ and $n$.

Let us say that $v\in V$ is {\em erratic} if $|S(v,r)|> 2d^r$. From~\eqn{eq:i}, we see that the expected number of erratic vertices is $o(1)$. Let us take any fixed $v$, expose the first $r$ neighbourhoods of $v$ out to $S(v,r)$, condition on $v$ not being erratic (which can be determined at this point), and  then fix $u \in S(v,r)$. Set $U=V\setminus  S(v,r)$.  We now expose the iterated neighbourhoods of $u$, but restricting ourselves to the graph induced by $ U$. Since $v$ is not erratic,  $| U|=n-o(\sqrt n)$.  Vertices found at distance $(r+1)$ from $u$ form a set $W(u) \subseteq S(u,r+1)$. We now argue   as in the derivation of~\eqn{badS}, but with $\eps=4/\log n$, and note that we are searching within a set of $n-o(\sqrt n)$ vertices. In this way it is easy to see  that with probability at least $1-o(n^{-2})$   we have $|W(u)| \sim d^{r+1}$.

Next we iterate the above argument, redefining $U$ to be the vertices not explored so far. In each case where the bounds on $|W(u)|$ hold for all steps, we have $|U| = n-o(n)$, since we stop when we have treated all of the at most $2d^r$ vertices in $S(v,r)$, and   $2d^{2r+1}\le 2 n/\log n$. Hence (ii) holds a.a.s.
\end{proof}

\section{Sparse case---deterministic result}\lab{s:sparse-det}

In this section we treat the sparse case, i.e.\ $  (1/2+\eps) \log n< p(n-1)< \log^3 n$. As for the dense case, we will first prove a general purpose result that holds for a family of graphs with some specific expansion properties. The next section will be devoted to show that sparse random graphs \aas\ fall into this class of graphs and so the conjecture holds \aas\ for sparse random graphs. Before stating the result, we need some definitions. 

A subset $U$ of $V(G)$ is {\em $(t,c_1,c_2)$-accessible} if we can choose a family $\{W(w):w\in U\}$ of pairwise disjoint subsets of $V(G)$ such that $W(w)\subseteq N(w,t)$ for each $w$, and
$$
|W(w)| \ge c_1\min\left\{  d^{t},\frac{c_2n}{  |U|}\right\}.
$$
This definition will be used for constants $c_1$ and $c_2$, and large $t$. The motivation is that, for an accessible set $U$, there are ``large'' sets of vertices $W(w)$ which are disjoint for each $w\in U$, such that any cop in $W(w)$ can reach  $w$ within $t$ steps.

Now we are ready to prove the upper bound on the cop number, provided that some specific expansion properties hold.

\begin{theorem}\lab{t:main}
Let $\G_n$ be a set of graphs and $d=d(n)\ge 2$. Suppose that $d < \log^J n$ for some fixed $J$ and that for some positive constants $\delta$ and  $a_i$ ($1\le i\le 5$), for all $G_n\in\G_n$ there exist $X(G_n) \subseteq V(G_n)$ with $|X(G_n)|=O(\sqrt n)$ such that  the following hold.
\begin{enumerate}
\item For all $v \in V(G_n) \setminus X(G_n)$, all $r \ge 1$ with $d^r < n^{1/2+\delta}$, all $r'$ that satisfy the same constraints as $r$, and all $V' \subseteq N(v,r)\setminus X(G_n)$ with $|V'|=k$ such that $k d^{r'}\le n/\log^J n$, we have  
$$
a_1k d^{r'}\le  |S(V',r')|\le   a_2k d^{r'}.
$$
In particular, with $k=1$
$$
a_1 d^{r'} \le |S(v,r')| \le a_2 d^{r'}.
$$
\item Let $r$ satisfy $   n^{1/4-\delta} < (d+1)d^r<n^{1/4+\delta}  $, and let $r'$ satisfy the same constraints as $r$. Let $v\in V(G_n)\setminus X(G_n)$, $A\subseteq S(v,r)\setminus X(G_n)$, and $U=\bigcup_{a\in A} S(a,r')$ with  $|A|> n^{1/4-\delta}$ and $ d^{r+r'}< a_3n/ |U|$.  Then there exists a set $Q$ such that $|S(a,r')\cap Q|<n^{1/4-2\delta}$ for all $a\in A$, and such that $U\setminus Q$ is  $(r+r'+1,a_4,a_5)$-accessible.  
\item $G_n - X(G_n)$ is contained in a component of $G_n$.
\end{enumerate}
Then $c(G_n) = O(\sqrt{n})$. 
\end{theorem}
\begin{proof}
First place a cop at each vertex in $X(G_n)$ to stay there throughout the game. These cops will not be referred to again, and their sole purpose is to permit us to assume that the robber never visits any vertex in $X(G_n)$. 

For a suitably large constant $F$, we create $i_f:=F \log \log n$ independent teams of cops that are chosen independently of each other. (For expressions such as $F \log \log n$   that clearly have to be an integer, we round up or down but do not specify which: the choice of which does not affect the argument.) Let $C$ be another sufficiently large constant to be determined later.  The $i$th team of cops  ($1 \le i \le i_f$) is determined by independently choosing each vertex $v \in V$ to be occupied by such a cop with probability $c_i/n$, where $c_i = C e^{-i} \sqrt{n}$ except that, for convenience, $c_{i}= \sqrt n$ when $i=i_f$.  (Note that it is possible that a vertex is occupied by several cops  from different teams.) It follows easily from Chernoff's bound~(\ref{chern}) that a.a.s.\ for all $i$, the $i$th team consists of $(1+o(1)) c_i$ cops.  This gives a total of $\Theta(\sqrt{n})$ cops a.a.s. We just have to show that the specified cops can a.a.s.\ catch the robber, where the probability involved is with respect to the random choice of cops.

Let  
$$
r_1 = \left\lfloor \frac 14 \log_{d} (\eps_0 n) \right\rfloor,
$$ 
where $\eps_0 >0$ is a sufficiently small constant (to be determined later). For $i \ge 2$, we recursively define $r_i$ to be an integer such that
\bel{rdef}
\frac {\sqrt{\eps_0 n}}{d } < \frac {d^{r_{i-1}+r_i}}{e^{2(i-1)}} \le \sqrt{\eps_0 n}.
\ee
Note that, once $r_{i-1}$ is fixed, there is exactly one integer that satisfies this property so the sequence $(r_i)$ is uniquely defined. Note that both $(r_{2j})_{j \in \N}$ and $(r_{2j-1})_{j \in \N}$ are nondecreasing sequences and $r_2 \ge r_1$, so $r_i \ge r_1 = \left\lfloor \frac 14 \log_{d} (\eps_0 n) \right\rfloor$ for every $i \ge 1$. It can happen that $|r_{i-1}-r_i|$ increases with $i$, but this causes no problem. The lower bound on $r_i$  just observed, combined with~\eqn{rdef},  gives 
\bel{dribound} 
d^{r_i} = \Omega(n^{1/4}/d),\quad  d^{r_i} =O(e^{2i}n^{1/4})
\ee
where the  constants implicit in $\Omega()$ and $O()$ depend on $\eps_0$.

The robber appears at some vertex $v_1 \in V$.   We consider up to $i_f$ rounds, where $i_f$ is defined above. In round $i$, the robber starts at $v_i$ and team $i$ of cops is released. Under conditions to be specified, team $i_f$ is also released in some earlier round.  Round $i$ lasts until the robber first reaches a vertex of $S(v_i,r_i)$ at which time round $(i+1)$ begins. As shown below, we can assume the robber eventually reaches a vertex of $S(v_i,r_i)$, that is, each round lasts for a finite time. Of course, round $i$ must last for at least $r_i$ steps.  

When each team of cops is released, they (or, to be more precise, many of them) will be given preassigned destinations to reach during the next two rounds. For instance, team 1 will be given the assignment of ``densely covering'' the sphere $S(u,r_2)$ for every single $u\in S(v_1,r_1)$, in a manner described below.  Inductively, at the start of round $i$ ($i\ge 2$), the cops from team $(i-1)$ are already heading towards some of the vertices in the sphere $S(v_i, r_i)$.  They have time to take precisely a further $r_i+1$ steps in which to get to their pre-assigned destinations, since even if the robber goes directly to $S(v_i, r_i)$, the cops reach their destinations on their very next turn.  If they reach this sphere before the robber does, the cops will just wait at their respective destinations for the robber to   ``catch up''.  We will show separately that   a robber cannot ``hide'' inside a sphere  forever. 

For $i\ge 1$, let $S_{i-1}$ denote the set of vertices in $ S(v_i, r_i)$ that are not protected by team $(i-1)$ (that is, not included in the preassigned destinations of those cops), so that in particular $S_0 = S(v_1, r_1)$. Then we may assume that the robber occupies a vertex in $S_{i-1}$ at the end of round $i$. We may assume that $ S_{i-1}$ contains no vertex of $X(G_n)$, since those are forbidden to the robber.  The cops in team $i$ will be assigned to vertices in the union of the sets $S(v, r_{i+1})$ over all vertices $v\in S_{i-1}$. 
The cops' strategy is to force the robber into  sets $S_i$ that have essentially decreasing size as $i$ increases. It is enough to show firstly that a.a.s.\ assignments can be made for the cops such that $S_i=\emptyset$ for some $i \le i_f$, and secondly that a robber cannot  hide  inside a sphere $S(v, r_i)$ forever. The final team $i_f$ is released at the start of the first round $i$ (possibly $i<i_f$) such that $|S_{i-1}| \le e^{-5(i_f-1)}  | S(v_i, r_i)  |$. We will show that there is a strategy to ensure that this inequality holds for $i\le i_f$, and that the final team of cops will be able to catch the robber. For this final team, we need to permit an extra round,  after round $i$  finishes, for them to reach their destinations. Teams $i, i+1,\ldots, i_f-1$ are not needed.

The hiding robber is the easiest aspect to deal with: by hypothesis (i) and~\eqn{dribound}, for each $1 \le i \le  i_f$  and any vertex $v \in V$, the ball of radius $r_i$ around $v$ has size at most $n^{1/4} e^{O(\log\log n)}$. So, as  in the dense case,  we can have one additional ``clean-up'' team  consisting of at most $ n^{1/3} $ cops. The clean-up team, at  the start of every round $i$, assign themselves one to each vertex of $N(v_i,r_i)\setminus X(G_n)$, and then start walking to their assigned positions, taking at most $n$ steps by assumption (iii). If the robber is still inside the ball after $n$ steps, the clean-up team are all in their positions and the robber must be caught. Hence, we may assume that each round lasts at most $n-1$ steps.

Recall that team $(i-1)$ of cops has, by definition of $S_{i-1}$,  a strategy to occupy all vertices in $S(v_i,r_i)\setminus S_{i-1}$, thereby effectively preventing the robber from going to these vertices at the end of round $i$.  We call a position for $v_i$ in round $i\ge 1$ \emph{vulnerable} if the set $S_{i-1}$ satisfies
\bel{eq:s_i}
|S_{i-1}| \le e^{-5(i-1)}  | S(v_i, r_i)  |.
\end{equation}
As $S_0 = S(v_1, r_1)$, the initial position $v_1$ is always vulnerable. 

The cops' strategy is, in general, to keep the robber at vulnerable positions. Observe that whether or not $v_i$ is vulnerable depends not only on $v_i$ and the initial (random) placement of team $(i-1)$ of cops, but also on the choices of strategies for the earlier teams of cops. The robber can see all the cops right from the first move of the game, and, by choosing appropriate steps, can potentially create many different possibilities for the set $S_{i-1}$. If any of these makes it possible to reach a $v_i$ that is not vulnerable, the cops' general strategy fails. So we have to be careful in defining and analysing the cops' strategy. 

For $u_1, u_2, \ldots , u_i\in V$, define $(u_1, u_2, \ldots , u_i)$ to be a {\em robber strategy to round $i$} if the robber can feasibly cause $v_j$ to equal $u_j$   for $1\le j\le i$, given that the robber can see all the cops from the start and knows their complete strategy.  To make the argument clearer, we develop the cops' strategy round by round. Given the teams of cops up to team $(i-1)$, assume we have defined a cop strategy for every robber strategy to round $(i-1)$. Now expose team $i$ of cops. For every robber strategy to round $i$, we will give a unique cop strategy for team $i$. If possible, the strategy makes $v_i$ vulnerable. Inductively,  since there are $n^{\Theta(i)}$ possible robber strategies to round $(i-1)$, there are  $n^{\Theta(i)}$ possible cop strategies using teams 1 to $(i-1)$, and hence  $n^{\Theta(i)}$ possible sets $S_{i-1}$. 
Observe that if $v_{i_f}$ is vulnerable, the condition for the release of the final round of cops is immediately met.

\begin{claim}\lab{clm2}
Let $i\ge 1$.  For a given vulnerable $v_{i }$ and a given possible $S_{i-1} \subseteq S(v_{i }, r_{i })$ satisfying~\eqn{eq:s_i} and such that $|S_{i-1}| > e^{-5(i_f-1)}  | S(v_i, r_i)  |$, the probability that team $i$ of cops has no strategy to cause  $v_{i+1}$  to be vulnerable is $O(n^{- \log n})$. 
\end{claim}

We first indicate how we will use the claim, once it is established. By the union bound, the probability that,  for each of the $n^{\Theta(i)}$ possible $v_i$ and sets $S_{i-1}$ corresponding to robber strategies,  team $i$ has  a   strategy that succeeds in making $v_i$ vulnerable, is $1-O(n^{\Theta(i)- \log n})$. Once round $i$ starts, of course $v_i$ and $S_{i-1}$ are determined, and the team can use the appropriate strategy.   Taking the union bound over all $i$ ($1\le i \le i_f-1=F\log \log n-1$), the probability that all teams have such strategies is $1-o(1)$. Thus, a.a.s.\ the teams can make $v_i$ vulnerable for all $i\le i_f$. Hence, by the observation above, the final team of cops will be released no later than the start of round $i_f$. After proving the claim, we will prove that the last team can finish the job a.a.s.
 
Next we prove Claim~\ref{clm2}. For $i\ge 1$, during rounds $i$ and $(i+1)$, the aim of team $i$ is to cover a large proportion of the sphere $S(u, r_{i+1})$ for every $u$ in $S_{i-1}$.  Note that the first move in a round belongs to the cops, and team $i$ can make one more move after the robber finishes  round $i$, to reach their desired destinations on the sphere, even though this move is technically the first cop move in round $(i+1)$. (At the same time, the next team makes their first move.) So any vertex $v$ can be covered by a cop in team $i$ whose distance is at most $r_i+r_{i+1}+1$ from $v$. We first check that hypothesis (ii) is satisfied with $A= S_{i-1}$, $r=r_i$ and $r'=r_{i+1}$ (which then define $U$).  Since $v_i$ is vulnerable, from~\eqn{eq:s_i} and (i), 
\bel{usize}
|U|\le |S_{i-1}| a_2d^{r' } \le   e^{-5(i-1)}a_2^2d^{r+r' } 
\ee
and hence by~\eqn{rdef}
$$
|U|d^{r+r'}\le e^{-5(i-1)} \eps_0 na_2^2e^{4(i-1)} = e^{-(i-1)} \eps_0 na_2^2.
$$
This verifies the condition $d^{r+r'}<a_3n/ |U|$, for a sufficiently small choice of $\eps_0$. The bounds on  $(d+1)d^r$ and on $(d+1)d^{r'}$ mentioned in (ii) follow  from~\eqn{dribound} and the fact that $d e^{2i} = \wO(1)$. The lower bound on $|A|$ required for (ii) follows from the lower bound on $|S_{i-1}|$ assumed in the claim, together with~\eqn{dribound}. Thus, by hypothesis (ii) there exists a set $Q$ with the specified properties. In particular, $U\setminus Q$ is  $(r+r'+1,a_4,a_5)$-accessible. Let $\{W(w):w\in U\setminus Q\}$ be the resulting family of disjoint sets. Then, for each $w$, $ |W(w)| \ge  a_4\min\left\{  d^{r+r'+1}, a_5n/ |U| \right\}$. Using~\eqn{usize} and~\eqn{rdef}, this implies  $|W(w)| = \Omega(1) e^{2i}\sqrt n$ where the constant in $\Omega()$ depends on $\eps_0$.
  
Now we can assign any cop from team $i$ that was originally placed in $W(w)$ to cover $w$. The probability that $W(w)$ contains no cop  is at most
\begin{eqnarray*}
\left( 1 - \frac {c_i}{n} \right)^{|W(w)|} 
&\le& \exp \left( - \frac {C e^{-i} \sqrt{n}}{n} \cdot  \Theta(1) e^{2i}\sqrt n \right)\\
 &\le& \exp \left( -  C \Theta(1)e^i   \right) \\
&<&   e^{-5i}/2,
\end{eqnarray*}
for $C$ sufficiently large. (Note that there is a {\em lot} of room in the argument at this point.) Consider a given $u\in U$. Since all the $W(w)$'s for the various $w\in H:=  S(u, r_{i+1}) \setminus Q$ are disjoint, the events that they are empty of cops are independent, each holding with probability at most $\frac 12 e^{-5i}$. By the lower bound on $|S(u, r_{i+1})|$ supplied by (i) with $k=1$, together with the upper bound $n^{1/4-2\delta}$ on $|S(u, r_{i+1})\cap Q|$ in assumption (ii), we have $|H| =  \Omega(n^{1/4-\delta})$.  So the expected number of vertices not covered is at most 
$$
|H| e^{-5i}/2 = \Omega(n^{1/4-\delta} / \log^{5F} n).
$$
A simple application of Chernoff's bound~(\ref{chern}) shows that with probability at least (coarsely) $1 - O(e^{- 2\log^2 n})$,  the proportion of $H$ that gets covered is at least $1-\frac 23 e^{-5i}$, and so (again recalling the upper bound on $|S(u, r_{i+1})\cap Q|$), all but at most an $e^{-5i}$ fraction of $S(u, r_{i+1})$ is covered. By the union bound, with probability $1-O(n^{- \log n})$, this holds for all $u$ in $S_{i-1}$. In particular, with probability $1-O(n^{- \log n})$,
$$
|S_{i}| \le e^{-5i} \big| S(v_{i+1}, r_{i+1}) \big|.
$$
This is the required condition~(\ref{eq:s_i})  for ${i+1}$, and thus $v_{i+1}$ is vulnerable. This proves the claim. 

\bigskip
 
It remains to consider the round $i$ in which the final team $i_f$ of cops is released. Recall, by the observation just after the statement of Claim~\ref{clm2}, that a.a.s.\   $i\le i_f$, and   $v_{i+1}\in S_{i -1}$. At this point, we know that
 $|S_{i-1}| \le e^{-5(i_f-1)}  | S(v_i, r_i)  |$, and this final team \aas\  contains $(1+o(1)) c_{i_f} = \Theta(\sqrt{n} )$ cops.
  The robber reqires $r_i$ steps  to  reach $ S_{i-1}$, and this time, we permit her $r''= r_{i_f}+r_{i_f+1}-r_i$ further steps. So her location at the end of this final round is restricted to
  $$
U =  \bigcup_{u \in S_{i-1}} S(u, r'') \setminus X.
$$
The cops in the last team have $\hat r = r_i+  r'' +1= r_{i_f}+r_{i_f+1}+1$ steps to do their job (recalling their ``free step'' after the robber's move). We will show that, with probability $1-O(e^{- \log^{3/2} n})$, this last team can cover {\em all} vertices in $U$
in $\hat r$ steps (not just a large proportion of each sphere, as before).  Once again, taking the union bound over the $n^{O(i_f)}$ sets $U$ feasible for the various robber strategies, a.a.s.\ there is no robber strategy preventing the cops from  catching the robber in this round.

We need to match suitable, distinct cops to the set $U$. By  Hall's theorem on systems of distinct representatives, it is enough to show that for each $V''\subseteq U $ with $|V''|=k$, there are at least $k$ cops within distance $\hat r$ of $V''$.  
 Using~\eqn{rdef},   we have   $  d^{\hat r}=O(e^{2i_f } \sqrt n)$. Hence, by (i) first with     $V'=\{u\}$ where $u \in S_{i-1}$, and secondly with   $V'=\{v_i\}$,
\begin{eqnarray*}
|U| &=& O(|S_{i -1}| d^{ r'' }) = O( e^{-5 i_f }  | S(v_i, r_i) |d^{\hat r-r_i} )\\
& = &O( e^{-3 i_f } \sqrt n \, | S(v_i, r_i) |d^{ -r_i} )= O( e^{-3 i_f } \sqrt n ).
\end{eqnarray*}
This immediately shows that the total number of cops is sufficiently large, and only their positions might cause a problem.
Moreover,  noting that  $U \subseteq N(v_{i },r_i + r'')= N(v_{i },\hat r-1)$, we may apply  (i) with $(r,r')$ defined by   $(\hat r-1,\hat r)$ and $V'=V''$. With a little foresight, we now set $F=J+2$. This ensures that the requirement  $kd^{r'}< n/\log^J n$ holds  in view of the above bound on $|U| $ (as $k\le |U|$) and~\eqn{rdef}.
  Thus, using~\eqn{rdef} once again, we obtain 
$$
 |S(V',\hat r)|=\Omega(   k  e^{2i_f}\sqrt n / d)=\Omega\big(k\sqrt n (\log n)^{2F-J} \big).
$$ 
(Recall that the cops in team $i$ are allowed to use   $r_i+r_{i+1}+1$ steps for their job.) We have $2F-J\ge 4$. 

It follows that the number of cops in the last team that occupy $\bigcup_{u \in V''} S(u, \hat r)$ can be bounded from below (for large $n$) by the binomial random variable $B(k   \log ^2 n \sqrt n  ,1/\sqrt n)$ with   expected value $\Omega(k \log^2 n)$. Using the Chernoff bound~(\ref{chern}) we get that the probability that there are fewer than $k$ cops in this set of vertices is less than $\exp(- \Omega(k \log^2 n))$. Hence, the probability that the necessary condition in the statement of the Hall's theorem fails for at least one set of vertices $V''$ is at most
$$
\sum_{k=1}^{|U|} {|U| \choose k} \exp( - \Omega(k \log^2 n) ) \le \sum_{k=1}^{|U|} n^k \exp( - \Omega(k \log^2 n) ).
$$
Thus, the perfect matching exists with probability  $1-O(e^{- \log^{3/2} n})$, as required.
\end{proof}

\section{Sparse case}\lab{s:sparse}

The purpose of this section is to verify that sparse random graphs \aas\ satisfy the conditions in the hypotheses of Theorem~\ref{t:main} and, as a result, Meyniel's conjecture holds \aas\ for this model. 

Before we state the key lemma let us introduce the following useful definition. For a given $\eps \in (0,1)$ we consider the function 
$$
f_{\eps}(x)=x \big(\log (\eps x)-1 \big).
$$ 
It is easy to show that $\lim_{x \to 0^+} f_{\eps}(x) = 0$, $f_{\eps}(x)$ is continuous and  decreasing on $(0,1/\eps]$ and $f_{\eps}(1/\eps)=-1/\eps<-1$. Therefore we can define $g(\eps)>0$ to be the unique value of $x\in(0,1/\eps]$ such that $f_{\eps}(x)=-1/2$. Moreover, since  $(\eps/e^2 \big )(2 \log \eps -3 \big)$ is decreasing for $\eps \in (0,1]$, we have  $f_{\eps}(\eps/e^2) = (\eps/e^2 \big )(2 \log \eps -3 \big) \ge -3/e^2 > -1/2$ and so $g(\eps) > \eps/e^2$.

Now we are ready to state the lemma. In several places we provide specific constants, but we make no attempt to optimise them.

\begin{lemma}\lab{lem:gnp sparse exp}
Let $0 < \eps <1$, $0 < \delta < \eps/6$ be two fixed constants. Suppose that $(1/2+\eps) \log n \le d=d(n)\le \log^3 n$. Then in $G = (V,E) \in \Gnp$ with $p=d/(n-1)$, the following property holds with probability $1-o(n^{-3})$.
\begin{enumerate}
\item For every vertex $v \in V$ and $r \ge 1$ such that $d^{r} < n/\log n$ 
$$
|S(v,r)| \le 9 d^r.
$$
\end{enumerate}
Moreover, a.a.s.\ there exists a set of vertices $D \subseteq V$ with $|D| \le \sqrt{n}$ such that the following properties hold.
\begin{enumerate}
\item [(ii)] For every $v \in V \setminus D$ and $r \ge 1$ such that $d^{r} < n/\log n$ 
$$
|S(v,r)| \ge \eps g(\eps) d^r (1+o(1)) > (\eps/e)^2 d^r.
$$
\item [(iii)] Let $r$ satisfy $   n^{1/4-\delta} < (d+1)d^r<n^{1/4+\delta}  $ and let $r'$ satisfy the same constraints as $r$. Let $v\in V \setminus D$, $A\subseteq S(v,r)$, and  $U=\bigcup_{a\in A} S(a,r')$ with  $|A|> n^{1/4-\delta}$ and $ d^{r+r'}< n/\gammaconst|U|$.  Then there exists a set $Q$ such that for each $a\in S(v,r)$,  $|S(a,r')\cap  Q|=O( d^{r'} n^{-1/54})$ and such that $U\setminus Q$ is  $(r+r'+1,1/50,1/9)$-accessible. 
\item [(iv)] For all $v \in V \setminus D$, all $r \ge 1$ with $d^r < n^{1/2+\delta}$, all $r'$ that satisfy the same constraints as $r$, and all $V' \subseteq N(v,r) \setminus D$ with $|V'|=k$ such that $kd^{r'} \le n / \log n$, we have 
$$
(\eps g(\eps)/4) k d^{r'} \le |S(V',r')| \le 9 k d^{r'}.
$$

\end{enumerate}
\end{lemma}

\begin{proof}
Put $x = 7.5$, so that ${x^2}/{(1+x/3)}>16$. For a given vertex $v \in V$, it follows from  Bernstein's bound (\ref{Bernstein}) that
\begin{eqnarray}\label{lowerexp}
\Prob \left( \deg(v) \ge d + xd \right) &\le&  \exp \left( - \frac {x^2 d} {2(1+x/3)} \right) \le n^{-4-8\eps}.
\end{eqnarray}
Thus, with probability $1-o(n^{-3})$ all vertices have degrees at most $(1+x)d$, and (i) holds for the case $r=1$.

We continue to estimate $|S(v,r)|$ for $r \ge 2$ in a similar way. We consider the exploration of the graph using breadth-first search (BFS) starting at the vertex $v$. We prove by induction that, for every $r \ge 2$ such that $d^r < n/\log n$,
$$
|S(v,r)| \le f(r) := (1+x) d^r \prod_{i=2}^r \left(1+ 2 {d^{-i/2}} {\log^{3/4} n} \right)
$$
with probability at least  $1 - o(n^{-4}) - r \exp(-\Theta(\log^{3/2} n))$. Note that $r = O(\log n / \log \log n)$, so this will finish the proof of (i), since all the statements will hold  for $v$  with probability $1-o(n^{-4})$, and
$$                       
\prod_{i=2}^r \left(1+ 2 {d^{-i/2}} {\log^{3/4} n} \right) \sim \prod_{i=4}^r \left(1+O \left( \log^{-1} n \right) \right) \sim 1.
$$
We may assume inductively that $|S(v,r-1)| \le f(r-1)$. Note, similar to the proof of Theorem~\ref{thm:dense_case}, that for $n$ sufficiently large
\begin{eqnarray*}
\lambda_r := \E |S(v,r)| &=& \left( 1 - \left( 1- \frac {d}{n-1} \right)^{|S(v,r-1)|} \right) \Big(n-|N(v,r-1)| \Big) \\
&\le&  \left( 1 - \exp \left(- |S(v,r-1)|\, \frac {d}{n} \left(1+\frac dn \right) \right) \right) n \\
&\le& |S(v,r-1)| d \left(1+ \frac {d}{n} \right) \\
&\le& f(r-1) d \left(1+ \frac {d}{n} \right)  \\
&=& O(d^r),
\end{eqnarray*}
provided $d^r < n/\log n$. Therefore, for $t_r=d^{r/2} \log^{3/4} n = o(d^r)$, Bernstein's bound (\ref{Bernstein}) gives
\begin{eqnarray*}
\Prob \left( |S(v,r)| \ge \lambda_r  + t_r \right) &\le& \exp \left( - \frac {t_r^2} {2(\lambda_r+t_r/3)} \right) \\
&\le& \exp \left( - \Omega \left(\frac{t_r^2}{d^r} \right) \right) = \exp \left( - \Omega (\log^{3/2} n) \right). 
\end{eqnarray*}
Hence, with desired probability,
\begin{eqnarray*}
|S(v,r)| &\le& f(r-1) d \left(1+\frac {d}{n}+ {d^{-r/2}} {\log^{3/4} n} \right) \\
&\le& f(r-1) d \left(1+ 2 {d^{-r/2}} {\log^{3/4} n} \right).
\end{eqnarray*}
This completes the inductive proof, and (i) follows.

\bigskip

We now turn to (ii).  We will not need  this as an independent conclusion. It is included in order to complete the picture and to present some arguments in a simple context which we will need to use for (iv).
We will prove that for $B \subseteq V$ such that $\eps g(\eps) d/2 < |B| = O(n/d \log n)$,  with probability at least $1 - \exp(- (\log^{3/2} n) / 4)$
\bel{partiii}
|N[B]| = |B| d \left(1+O(\log^{-1} n) + O \left(\frac {\log^{3/4} n}{(|B|d)^{1/2}}\right) \right)\sim  |B| d.
\ee
In particular, $|N[B]| = (1+O(\log^{-1} n)) |B| d $ if $|B| \ge d^{5/2}$. 

In order to prove~\eqn{partiii}, let us observe that for any $B \subseteq V$ such that $\eps g(\eps) d/2 < |B| = O(n/d \log n)$ we have 
$$
\lambda := \E (|N[B]|) = |B| d (1+O(\log^{-1} n)),
$$ 
and it follows from Chernoff's bound~(\ref{chern}), by taking $\bar{\eps} = (|B|d)^{-1/2} \log^{3/4}n$, that
$$
\Prob \left( \big| |N[B]| - \lambda \big| \ge \bar{\eps} \lambda \right) \le 2 \exp \left( - \bar{\eps}^2 \lambda/3 \right) \le \exp \left( - ( \log^{3/2} n) / 4 \right). 
$$
Hence~\eqn{partiii} holds with  probability at least $1- O( \exp ( - (\log^{3/2} n)/4 ))$.  Moreover, we note the following.

\begin{observation}\lab{for4.2} 
Let $B$ satisfy the conditions applying to~\eqn{partiii} and let $Z \subseteq V \setminus B$ be a set of cardinality $O(n/\log n)$. Then, conditional upon all the edges that do not join  vertices in $B$ to vertices in $V \setminus (B\cup Z)$, with probability at least $1-\exp(-\Omega(\log^{3/2}n) )$, \eqn{partiii} holds with $N[B]$ replaced by $N[B] \setminus Z$.
\end{observation} 

\noindent The truth of this is clear, considering the proof of~\eqn{partiii}, since after deleting $Z$  the number of  vertices available to join to $B$ is still a fraction $(1+O(\log^{-1}n))$ of the original.
 
Now, let us come back to the proof of part (ii). Note that (i) estimates an upper bound of the type implicit in~\eqn{partiii} but for $|B|=1$. Moreover, if $d < (1-\eps) \log n$, then a.a.s.\ there are isolated vertices, so there is no hope for a non-trivial lower bound. However, it is possible to exclude at most $\sqrt{n}$ vertices to obtain a bound that matches the order of the upper bound a.a.s. For this we use Chernoff's bound (\ref{chern-strong}), which gives
\begin{eqnarray*}
\Prob( \deg(v) \le \eps g(\eps) d) &=& \Prob \Big( \deg(v) \le d - (1-\eps g(\eps)) d \Big) \\
&\le& \exp \Big( -d \psi(-(1-\eps g(\eps))) \Big) \\
&\le& \exp \left( - \left(\frac {1}{2} + \eps \right) \psi(-1+\eps g(\eps)) \log n \right) = o(n^{-1/2}),
\end{eqnarray*}
since
\begin{eqnarray*}
  \psi \big(-1+\eps g(\eps) \big)  =     1+ \eps g(\eps) \big( \log( \eps g(\eps) ) - 1 \big)  =   1 - \frac {\eps}{2} 
\end{eqnarray*} 
for $\eps \in (0,1)$ (using the definition of $g(\eps)$). Hence, the expected number of vertices of degree at most $\eps g(\eps) d$ is $o(\sqrt{n})$, and so a.a.s.\ there are at most $\sqrt{n}$ such vertices by Markov's inequality. These vertices we define to be the set $D \subseteq V$. For any vertex in $V \setminus D$ (that is, of degree larger than $\eps g(\eps) d$), it follows from Observation~\ref{for4.2}, with the excluded set $Z$ being, in each application, a neighbourhood $N(v,r-1)$, that the successive neighbourhoods $N(v,r)$ for increasing values of $r$ will expand regularly (that is, the sizes of neighbourhoods are within the bounds specified in~\eqn{partiii}), up until reaching a neighbourhood of size at least $n/d\log n$. Part (ii) now follows by induction on $r$, noting that the accumulated error factor is $(1+o(1))$ by reasoning similar to that for  part (i).

\bigskip

Next we will prove part (iii).  Fix $v \in V \setminus D$, and fix $r, r'$ satisfying the properties stated in part (iii). We will show that with probability $1-o(n^{-3})$ the desired property holds for every $A \subseteq S(v,r)$, and the result will follow by the union bound. This argument is longer and will be broken up into several phases. 

\medskip
\noindent{\em Phase 1: run the exploration process for $t'$ rounds}
\smallskip

\noindent Consider the process of exploring the random graph by exposing edges to determine  successive neighbourhoods of $v$ (in breadth-first search manner) discussed in the proof of (i). For $i \ge 0$ let $L_i$ denote the vertices at distance $i$ from $v$. We say that vertices of $L_i$ become \emph{exhausted} during round $i$ and, as a result, a number of vertices are found, which we call \emph{pending}, that will become exhausted in the next round. 
(In the very unlikely event that the number of pending vertices ever drops to 0 prematurely, i.e.\ a complete component is discovered before all $n$ vertices are reached, the next step can be to just choose another vertex at random and nominate it as ``pending''. We make this convention just so that the events we define make sense.)

We continue the process until the first complete round in which the total number $n'$ of vertices that have been encountered is at least $n/d^3\log^2 n$. (Note our convention $a/bc=a/(bc)$.) Let $t'$ be the index of this round. We may apply Observation~\ref{for4.2} to deduce that, with probability $1-o(n^{-3})$, 
\bel{nsizegnp}
n/d^3\log^2 n\le n'\le (1+o(1)) n/d^2\log^2 n.
\ee
Define $U_0=\bigcup_{u\in S(v,r)} S(u,r')$. 

Let $G_0$ be the graph induced by the vertices reached in the process up to this point. The first part of our strategy will be to find, with probability $1-o(1/n^3)$, some large disjoint sets, associated with each vertex in a set containing almost all vertices in $U_0$. These sets will then be ``grown'' outside $G_0$ where necessary to form the larger sets $W(w)$, for each $w\in U$. The probability that they cannot be grown with their desired properties, for any particular $U\subseteq U_0$, will be so small that a union bound over all $U\subseteq U_0$ will yield the desired bound on the probability that the required sets $W(w)$ exist. 
 
\medskip
\noindent {\em Phase 2: re-examine the process to round $r+r'+1$}
\smallskip

The exploration process, performed in the BFS manner, has revealed at this point a tree $\widehat T$ rooted at $v$. Note that each vertex $w\in U_0$ was reached in this process by the time of completion of round $r+r'$. We ``rewind'' the process and ``play'' it one more time from the beginning to the end of round $r+r'+1$.  For  each vertex $w \in U_0$, we will use $T_j(w)$ to denote the subtree of $\widehat T$ rooted at $w$ of height $j$. Formally, for a given vertex $w\in U_0$, let $\tau$ denote the distance from $v$ to $w$, so that $w\in L_{\tau}$. Set $L_0(w) =\{w\}$, and inductively for each $i\ge 0$ denote by $L_{i+1}(w)$ the set of vertices in $L_{\tau+i+1}$ adjacent to vertices in $L_i(w)$ in the BFS tree $\widehat T$. In other words, $L_{i+1}(w)$ comprises the vertices that are found in the BFS process, and become pending, while vertices in $L_i(w)$ are being exhausted. Then set $T_j(w)= \bigcup_{i=0}^ {j} L_i(w)$. With a slight abuse of notation, we will also use $T_j(w)$ to denote the subtree induced by this set of vertices. 

We now define the set $Q$ by placing each $w \in U_0$ into $Q$ if and only if $X(w): = |T_1(w)| < 2d/3$.  (Note that usually we would expect $X(w)\sim d$.) It follows from part (i) that with probability $1-o(n^{-3})$ we have the event, call it $H_1$, that every vertex has at most $9 d^{r+r'+1}$ vertices at distance $r+r'+1$. Unfortunately, during the exploration process we cannot rely on this observation because conditioning on $H_1$ would not permit the edges to be independent.  To avoid this problem, and to make it easy to bound the size of $Q$ regardless of whether $H_1$ holds, we exclude from $Q$   every vertex whose neighbours are not fully revealed before $9 d^{r+r'+1}$ vertices in the set $S(v,r+r'+1)$ have been discovered. Consequently, there will always be  $n-O(n^{0.9})$, say, vertices available when edges are still being exposed and $Q$ is still being generated. Observe that, as long as $H_1$ is true, $Q$ contains all the low degree vertices we are concerned with at present.

We next bound the size of $Q$ in the typical cases. We can clearly couple the variables $X(w)$ with independent variables $Y(w)$ each with distribution ${\rm Bin}(n - n^{0.91}, p)$, by using $Y(w)$ to determine the edges from $w$ to $n - n^{0.91}$ of the unexplored vertices at each step (where quantities like $n - n^{0.91}$ can have either floor or ceiling inserted). Then   $Y(w) \le X(w)$ for each $w$,  and from Chernoff's bound~(\ref{chern})  
\begin{eqnarray*}
 \Prob \Big( Y(w) < 2d/3 \Big)  &\le& \Prob \Big( |\E Y(w) - Y(w) | \ge (1-\eps/3) \E Y(w)/3 \Big) \\
&\le& 2 \exp(- (1-\eps/3)^2 \E Y(w) / 27) \le n^{-1/54},
\end{eqnarray*}
since $(1-\eps/3)^2 (1/2+\eps) > 1/2$ for $0 < \eps < 1$. Hence for each $a\in S(v,r)$,  $|S(a,r')\cap  Q|$ is stochastically bounded from above by the binomial random variable ${\rm Bin}(9d^{r'}, n^{-1/54})$, and we have, for some particular choice of the constant implicit in $O()$, with probability $1-o(1/n^4)$ 
$$
|S(a,r')\cap  Q| \le 2 (9d^{r'}) n^{-1/54} = O(d^{r'} n^{-1/54}).
$$ 
Hence this bound holds for all $ a\in S(v,r)$ with probability $1-o(1/n^3)$. Additionally, since $H_1$ fails with probability $o(n^{-3})$, we may add to these upper bounds the condition that $Q$ contains all vertices of low degree within distance $r+r'$ of $v$. 
 
For (iii), it only remains to show that, with probability $1-o(n^{-3})$,  for all appropriate sets $U$, the set $U\setminus Q$ is  $(r+r'+1,1/50,1/9)$-accessible.  (Let us recall that (iii) is only required to be true {\aas} We aim for probability $1-o(n^{-3})$ to prepare the way for using the union bound over various $v$, $r$, and $r'$.)
 
Note that $ L_{r+r'}\subseteq U_0$, but that other vertices of $U_0$  are scattered at various distances from $v$. We first consider $\hatU_0:=    L_{r+r'}\setminus Q \subseteq U_0$.   We have from above that the trees $\{ T_1(w): w\in \hatU_0 \}$ are pairwise disjoint trees with at least $(2/3)d$ leaves. Moreover, it follows from part (i) that we may assume the number of leaves in each tree to be at most $9d$. These trees are all based at the same level, which simplifies the presentation of our analytic arguments.  

\medskip
\noindent{\em Phase 3: re-examine  from round $r+r'+1$ to round $t'$.}
\smallskip

\noindent We next re-examine the exploration process from level $r+r'+1$, and extend each tree $T_1(w)$ ($w\in\hatU_0$)  into a tree $\tilde T(w)$ that reaches ``up'' as far as vertices in  $L_{t'}$. This is done in a BFS manner as before, adding one level to all of the trees  before continuing to the next level. 

Since for each $w \in \hatU_0$ we have $(2/3)d \le |T_1(w)| \le 9d$ and the number of vertices discovered is at most $n' < n/\log n$, it follows from Observation~\ref{for4.2} (with the excluded set $Z$ being $N(v,r+r'+1)$ together with all the trees that are already grown) that the trees grow in an approximately regular fashion. To be precise, with probability $1-o(1/n^3)$,  the trees $\tilde T(w)$ can be defined for all $w\in \hatU_0$, so that they have the following property.

\begin{observation}\lab{obsgnp} 
The trees $\tilde T(w_1) $  and $\tilde T(w_2) $  are pairwise disjoint for $w_1,\, w_2\in \hatU_0$. Furthermore, each tree $\tilde T(w)$ contains at most $9d^{t'-r-r'} (1+o(1))$ and at least $(2/3)d^{t'-r-r'} (1-o(1))$ vertices at the top level, i.e.\ in $L_{t'}$. 
\end{observation}

\medskip
\noindent{\em Phase 4: the exploration process after round $t'$}
\smallskip

\noindent
We can now condition on the so-far-exposed subgraph $G_0$ of $G$ satisfying the event shown in the Observation~\ref{obsgnp} (for some specific choice of the functions hidden in the $o()$ notation).  Let  $A\subseteq S(v,r)$ with $|A|> n^{1/4-\delta}$ and $ d^{r+r'}< n/\gammaconst|U|$. Since there are at most $2^{n^{1/4+\delta}}$ choices of the set $A$  (due to the fact that $|S(v,r)| \le n^{1/4+\delta}$)  and $n$ choices of $w$, we are done by the union bound once we show that the probability that for this set $A$, the required sets $W(w)$ exist with the desired properties with probability $1-o(2^{-n^{1/4+\delta}}/n^3)$. 
 
Let $U=\bigcup_{u\in A} S(u,r') \subseteq U_0$. We will grow the trees $\tilde T(w)$ a little higher, from their present height $t-r-r'$ to height $r+r'+1$ for every $w \in U \cap \hatU_0$. (Recall that $\hatU_0 = L_{r+r'}\setminus Q \subseteq U_0$.) For such vertices $w$, the set $W(w)$ will be chosen from the vertices of $\tilde T(w)$. Afterwards, to cope with vertices $w \in (U\setminus Q) \setminus \hatU_0$, trees will be grown from some vertices of $\hatU_0$ to different heights, and a single set $W(w)$ may contain vertices of several trees $\tilde T(w')$ where $w'\in \hatU_0$.

First, let us grow the trees $\tilde T(w)$ to height $r+r'-1$ for every $w \in U \cap \hatU_0$, which is two steps short of our target height. We will, if necessary, prematurely terminate the process to make sure that each tree has at most $9 d^{r+r'-1}$ vertices.  Since $|U| 9 d^{r+r'-1} < n/d $ and the size of $G_0$ is $o(n /d)$ by~\eqn{nsizegnp}, this guarantees that the number of vertices available during this phase of the process is always at least $n(1-2/\log n)$. Moreover, each tree originally contains at least
$$  
(1-o(1)) \frac 23 d^{t'-r-r'} \ge n^{1/2-2\delta}/\wO(1)
$$
vertices on the top level (by Observation~\ref{obsgnp}). Hence, once again applying Chernoff's bound~(\ref{chern}), we deduce that with probability $1-O(2^{-n^{1/2-3\delta}})$ all the trees grow by a factor of $(1+O(\log^{-1} n)) d$ in each step, and so each tree has at least $(1+o(1)) (2/3) d^{r+r'-2}$ vertices on level $r+r'-2$. We then grow them one more step, which could potentially increase their size by another factor of $(1+o(1))d$. However we ensure that each tree $\tilde T(w)$,  by  terminating its generation prematurely if necessary, has $(1+o(1)) \frac {2}{3} d^{r+r'-1}$ vertices at distance $r+r'-1$ from $w$. 

Let us now grow the trees another step (the second-last) to height $r+r'$. This time, all the trees will grow by a factor of at most $(1+o(1)) d$, but some may grow less. In order to keep the tree sizes balanced, the next layer is grown but the process of expanding a given tree is terminated prematurely if it reaches $d^{r+r'}/2$ vertices. Arguing as before, this guarantees that at least $n - |U|d^{r+r'}/2 - o(n) > (17/18-o(1)) n$ vertices are always available. With all but negligible ($O(2^{-n^{1/2-3\delta}})$) probability, each tree can be grown by a factor of $(17/18-o(1))d$ which implies that with all but negligible probability, each $\tilde T(w)$ has $(1+o(1)) d^{r+r'}/2$ vertices at distance $r+r'$ from $w$. (Note that $(2/3)(17/18)>1/2$.)

Finally, the trees are grown for the last step. Let us recall that our goal is to show that there exists a constant $c_1>0$ such that $U\setminus Q$ is $(r+r'+1,c_1,1/\accessconst$)-accessible; that is, we need to construct a family $\{W(w):w\in U \setminus Q \}$ of pairwise disjoint subsets of $V(G)$ such that $W(w)\subseteq N(w,r+r'+1)$ for each $w$, and
$$
|W(w)| \ge  c_1 \min\left\{  d^{r+r'+1},\frac{n}{\accessconst |U|}\right\}.
$$
(In fact, for now we focus on $U \cap \hatU_0$ before showing that the property holds for $U$.) Because of the minimum function at the lower bound for the size of $W(w)$, let us independently consider the following two cases. 

Suppose first that $d^{r+r'+1} \le n/9|U|$.  This time, we terminate the process prematurely if it reaches $d^{r+r'+1}/3$ vertices.  Arguing as before, at most 
$$
|G_0| + \frac 13 d^{r+r'+1} |U| < (1+o(1)) \frac {n}{3 \cdot 9}= (1+o(1)) \frac {n}{27}
$$
vertices are reached at the end of this process, and so all the trees can grow by another factor of $(26/27-o(1)) d$. Hence, with all but negligible probability, we terminate the generation of each tree prematurely to get $(1+o(1)) d^{r+r'+1}/3$ leaves in each one. The desired property is obtained with $c_1=1/4$.

Suppose now that $d^{r+r'+1} > n/9|U|$. We grow all the trees one by one, terminating the process for a given tree prematurely once we reach $n/9|U|$ leaves, at which point we move on to the next tree. (This time, the premature termination for a given tree is not necessarily likely, as we shall see below.) Since we will discover at most $(1+o(1)) n/9$ vertices during this final step, the number of vertices available is always at least 
$$
n - |G_0| - (1+o(1)) |U| \frac {d^{r+r'}}{2} - (1+o(1)) \frac {n}{9} =  (1+o(1)) \frac {5}{6} n. 
$$
Therefore, with all but negligible probability, for each tree we either stop the generation process prematurely to get $(1+o(1)) n/9|U|$ leaves or form a set of leaves of cardinality at least 
$$
(1+o(1)) \frac {1}{2} d^{r+r'} \frac {5}{6} d = (1+o(1)) \frac {5}{12} d^{r+r'+1} \ge \frac {1}{4} \cdot \frac {n}{9|U|}.
$$
The desired property is obtained with $c_1=1/4$ as in the previous case. In both cases, we put the leaves of $\tilde T(w)$ into $W(w)$, and the desired property for $W(w)$ ($w \in U \cap \hatU_0$) holds with $c_1 = 1/4$. 

It remains to show that appropriate sets $W(w)$ can be defined for $w \in U \setminus \hatU_0$, that is, for vertices of $U$ that are ``buried'' inside the sphere $S(v,r+r')$. Consider $w \in R_j = L_{r+r'-j} \cap U \setminus Q$ for some $j \ge 1$.  In Phase 2  the tree $T_{j+1}(w)$, rooted at $w$, was defined. This tree reaches up to the layer $L_{r+r'+1}$. In that phase, we were able to assume that the event $H_1$ holds because it fails with  probability $o(n^{-3})$. We are similarly permitted to assume that an event holds which has probability $1-o(n^{-3})$ of occurring in the BFS process initiated at $v$ up to the layer  $L_{r+r'+1}$. In particular, since $w \notin Q$, with probability crudely bounded by $1-o(n^{-4})$, the last layer of this tree (subset of $L_{r+r'+1}$) has at least $(1+o(1)) (2/3) d^{j+1}$ vertices by Observation~\ref{for4.2}. Hence, this property holds for every $w \notin Q$ with probability $1-o(n^{-3})$. When this is true, it  implies that one layer lower, in the set $T_{j+1}(w) \cap L_{r+r'}$, there are plenty of vertices that are not in $Q$. Indeed, we may assume that each vertex has degree at most $9d$ (since this holds with probability $1-o(n^{3})$ by part (i)) so there must be at least $(1+o(1)) (2/27) d^{j}$ vertices in $T_{j+1}(w) \cap L_{r+r'}$ that are \emph{not} in $Q$. (We note that a factor of 1/9 is lost because of this simple worst-case argument, but this causes no problem.) Let $F(w)=T_{j+1}(w) \cap L_{r+r'} \setminus Q$. We will re-use $\tilde T(w')$ for vertices $w' \in F(w)$ that were created during phase 3 (see Observation~\ref{obsgnp}) and grow them (if necessary) to height $r+r'+1-j$. Some of these trees are already grown to height $r+r'+1$ (this is the case when $w'  \in U$), but some of the others  will need to be extended (though only, as we shall see, for certain small values of $j$).

Let $F$ denote the union of the sets $F(w)$ over all $w \in \bigcup_{j \ge 1} R_j$. For a given vertex $w' \in F$ let $j=j(w')$ be the minimum positive integer with the property that there exists $w \in R_j$ such that   $w' \in F(w)$. We can condition on $G_0$ having trees $\tilde T(w)$ for all $w \in \hatU_0$ with the property in Observation~\ref{obsgnp}. That is, these are disjoint trees based on all vertices in $\hatU_0$, each of height $t'-r-r'$.

The trees $\tilde T(w')$ are now grown further, up to the required heights, for all $w' \in F\cup (U \cap \hatU_0)$, treating each such $w'$ in turn, and in the manner described earlier. Actually, if $w' \in F$ has $j=j(w')$ such that $r+r'+1-j \le t'-r-r'$, the tree $\tilde T(w')$ already has sufficient height and does not need to be grown any further. Hence, we may assume that $r+r'+1-j > t'-r-r'$. The process goes exactly as discussed earlier. In particular, in the final two steps, the generation process is terminated prematurely as before to obtain the desired bound for the number of vertices. The argument still applies and we get that for any $w'$ with $j=j(w')$ the number of leaves in $\tilde T(w')$ is at least $\frac {1}{3d^{j}}  \min\left\{  d^{r+r'+1},\frac{n}{\accessconst |U|}\right\}.$ As a consequence, each vertex $w \in R_j$ has at least 
$$
(1+o(1)) (2/27) d^{j} \frac {1}{3d^{j}}  \min\left\{  d^{r+r'+1},\frac{n}{\accessconst |U|}\right\} \ge \frac {1}{50}  \min\left\{  d^{r+r'+1},\frac{n}{\accessconst |U|}\right\} 
$$
vertices at distance $r+r'+1$ from $w$. These vertices form set $W(w)$. Since the trees have disjoint level sets, $W(w_1)$ and $W(w_2)$ are disjoint whenever $w_1$ and $w_2$ are not in the same set $R_j$, whilst if they are in the same $R_j$, we have $F(w_1) \cap F(w_2) = \emptyset$ and thus $W(w_1)$ and $W(w_2)$ are disjoint as well.  This completes the proof of (iii).
\medskip

Now consider part (iv).  
If the property in part (i) is true, then the upper bound in (iv) immediately holds deterministically by restricting (i) to all $v'\in V'$. Hence, we need to focus on the lower bound only. We will show that for each of the $O(n \log n)$ ways to choose $v$ and $r$,  with probability $1-o(1/n \log n)$ no set $V' \subseteq N(v,r)$ under consideration fails the desired property for any $r'$.  For this, we use arguments that are mainly very similar to those in part (iii) but in a slightly simpler setting, so we are a little  less explicit in the details.

Fix a vertex $v\in V(G)$. Consider the  BFS exploration  process starting from $v$, but now consider processing the pending vertices one at a time, each time exposing the neighbourhood of the active vertex. Stop the process after finding the neighbourhood $N(v,s)$,  where $s$ is minimum such that $|N(v,s)|>n^{2/3}$.  Let $T$ denote the BFS tree restricted to $N(v,s)$.     Note that, in view of (i), we will be able to assume that $T$ has  $O(n^{2/3} d) = O(n^{2/3+o(1)} )$ vertices and that $d^s=n^{2/3+o(1)}$; to express this in a  technically correct manner requires defining intersections of the event that this is true, with any other events under consideration, as we have done several times before. This is quite straightforward but a little  tedious, so we just assume these statements hold deterministically henceforth. 
Let $H(v)$ be the property that for each vertex $w$ in $N(v,s-1)$,  there are at most 9 edges in $G$ (counting both exposed and unexposed edges) from $w$ to other vertices that were  in  $T$  at the start of the step of exposing the neighbourhood of $w$ as a pending vertex. Then $H(v)$ holds with probability at least $1-o(n^{-2})$ (this probability suffices for our present purposes) for the following reason. For a given vertex $w$, the probability that there are at least ten  edges from $w$ to other vertices of $T$ already reached is at most 
$$
{n^{2/3+o(1)} \choose 10} \left( \frac d{n-1} \right)^{10} = n^{-10/3 + o(1)},
$$
since $d \le \log^3 n$, so the claim holds by the union bound. When $H(v)$ holds, each vertex in $N(v,s-1)$ that is not in $D$ has at least $\eps g(\eps)d-9$ children in $T$.

Extending this idea, consider the property $H^+(v)$   that for all  $s'+1\le s''\le s$ and each vertex $w$ in  $S(v,s')\setminus D$, the number of descendants of $w$ that are at distance $s''$ from $v$  is at least  $\eps g(\eps)d^{s''-s'}/2$. Then the case   $s''=s'+1$ is implied by $H(v)$, and for larger $s''$, Observation~\ref{for4.2} can easily be applied inductively to this set of descendants, to show that $H^+(v)$ holds with the desired probability  say, $1-o(n^{-2})$.
   
Now consider $V' \subseteq N(v,r) \setminus D$ with $|V'|=k$  under the constraints given in (iv). For all $r'$ such that $r+r'\le s$, the condition required in the lemma is implied by $H^+(v)$, because the descendants of $w$ contained in $N(v,s'')$ with $s''=r+r'$ are contained in $S(w,r')$, and all these sets are disjoint for different vertices $w$. Note that, since $d^s=n^{2/3+o(1)}$ and $d^r<n^{1/2+\delta}$ where $\delta<1/6$, we may assume that the descendants of $V'$  for many generations, to be specific  at least ten generations, are not leaves of $T$. In particular, we may assume that   $r+r'\le s$ if  $r'\le 10$.  

To cover all relevant $r' > 10$,  we extend the exploration process to create a super-tree $T'$ of  $T$, but using the following variation of the BFS paradigm: any pending vertex that has distance at least $r'$ from all vertices of $V'$, with distance measured in the growing tree $T'$, is artificially declared exhausted and its neighbours are not explored. Note that this defines potentially a different process for each set $V'$ and each $r'$. 

Let $R_i$ denote the set of vertices in $V'$ whose distance from leaves of $T$ is $i$; i.e., the distance from $v$ is $s-i$. Call $i$ {\em good} if $|R_i|\ge k/\log^2 n$. Let  $D_i(j)$ denote the set of descendants of vertices in $R_i$ in $T'$ at depth $j$. 
 
Suppose that $kd^{r'} \le n / \log n$.  We are at liberty to further restrict the BFS-type process generating $T'$ so that when it is processing pending vertices at a given level, all pending descendants of vertices in $R_i$ are processed before moving on to descendants of $R_{i'}$ for some $i'\ne i$. Under these conditions we claim that, with probability  $1-O\big(\exp(-k\log^2 n)\big)$, either $H^+(v)$ fails or    
\bel{Rbound}
|D_i(r')| \ge   \big(1-O(r'/\log n)\big) \eps g(\eps)|R_i| d^{r'} /2 
\ee
for every good $i$, $1\le i\le r'$.

We first point out why this is good enough for our purposes. Since $d=\Omega(\log n)$, we know that $r'=O(\log_d n)= o(\log n)$   and hence the sum of $|R_i|$ over all good $i$ is $k -o(k/\log n)$. So~\eqn{Rbound} implies that the vertices in $V'$ have at least  $\big(1-o(1)\big) \eps g(\eps)k d^{r'} /2 $ descendants at distance $r'$.  The descendants at distance $r'$ from the various elements of $V'$ are by definition disjoint sets. Hence 
$$
 |N(V',r')| > (\eps g(\eps)/3) d^{r'} k 
$$
with probability  at least $1-\exp\big(-\Omega(k\log^2 n)\big)$. 
 Taking the union bound over at most ${n\choose k}=\exp\big(O(k\log n)\big)$ choices for a set $V'$ of cardinality $k$, we deduce that this bound   a.a.s.\ holds simultaneously for all such sets  $V' \subseteq N(v,r) \setminus D$ (or $H^+(v)$ fails, which we know is a.a.s.\ false).
To deduce  (iv) from this, we only need that when (i) holds, $|S(V',r')|\sim|N(V',r')|$ for all such $V'$. 

It only remains to show~\eqn{Rbound}. Note that $s-r = \Omega(\log n / \log \log n)$ because $d\le (\log n)^{O(1)}$, and hence the first ten generations (at least) of the descendants of $R_i$ are inside $T$. So $R_i$ is empty for $i\le 10$. For $i\ge 10$, we can deduce~\eqn{Rbound} by induction on $r'$.  Firstly, $H^+(v)$ implies that the generations descending from $R_i$ have the correct size until the leaves of $T$ are reached. Since this is at least ten generations,   we know there are   $\Omega(d^{10} (k / \log^2 n)) =\Omega(k \log^8 n)$  of them, as $i$ is good. Conditional $D_i(j-1)$ ($j\ge 11$), we have $\ex|D_i(j)|= d|D_i(j-1)|\big(1-O(1/\log n)\big)$ provided this quantity is $O(n/\log n)$. We now use~\eqn{chern} again, with $\eps=1/\log n$ and assuming inductively that $d|D_i(j-1)|=\Omega(k\log^8 n)$, to deduce that 
$$
|D_i(j)|= d|D_i(j-1)|\big(1-O(1/\log n)\big)
$$
with probability $1-\exp\big(-\Omega(k\log^2 n)\big)$. This implies~\eqn{Rbound} by induction on $j\le r'$.
\end{proof}

 Finally, we are ready to show that Meyniel's conjecture holds for sparse random graphs.

\begin{theorem}\lab{t:lowdreggnp}
Let $0 < \eps <1$, and suppose that $(1/2+\eps) \log n \le d=d(n)\le \log^3 n$. Let $G  \in \Gnp$ with $p=d/(n-1)$. Then a.a.s.\ 
$$
c(G) = O(\sqrt{n}).
$$ 
\end{theorem}
\begin{proof}
We will use Lemma~\ref{lem:gnp sparse exp} to show that $G$ \aas\ satisfies the conditions in the hypotheses of Theorem~\ref{t:main} with $J=3$. The set $D$ in Lemma~\ref{lem:gnp sparse exp} will play the role of $X(G_n)$ in Theorem~\ref{t:main}.

Condition (i) of Theorem~\ref{t:main} follows directly from Lemma~\ref{lem:gnp sparse exp}(iv) with $a_1 = (\eps g(\eps)/4)$ and $a_2=9$. Condition (ii) follows from Lemma~\ref{lem:gnp sparse exp}(iii) with $a_3 = 1/9$, $a_4=1/50$, $a_5 = 1/9$ and any $ \delta < 1/162$.  (Note that Lemma~\ref{lem:gnp sparse exp}(iii) is in fact slightly stronger than (ii); there is no need to remove $X(G_n)$ from $A$.)

In order to check condition (iii), take any two vertices $v,w \in G_n - X(G_n)$ in $G  \in \Gnp$, and investigate their neighbourhoods out to distance $r = \lceil (2/3) \log_d n \rceil$. Let us condition on these neighbourhoods satisfying the inequalities in   Lemma~\ref{lem:gnp sparse exp}(ii), and also on   $N(v,r) \cap N(w,r) = \emptyset$. Then, using the Chernoff bound as usual, we see easily that with probability $1-o(n^{-2})$ there is at least one edge joining $S(v,r)$ to $S(w,r)$, and so with this probability $v$ and $w$ belong to the same component. Hence, \aas\ all vertices in $G_n - X(G_n)$ for which the condition in Lemma~\ref{lem:gnp sparse exp}(ii) holds are in the same component. In particular this applies to the vertices of $V\setminus D$. This shows (iii), and the theorem follows.
\end{proof}

\section{Final remarks}

First of all, note that only the special values of $p$ in Case 2 in the proof of Theorem~\ref{thm:general_dense_case} required much care, and it is reasonable to suppose that a bit more work, perhaps combining our approach with that in~\cite{lp2}, would easily produce a sharper result in Theorem~\ref{thm:zz}(ii), namely, that the upper bound  can be made a constant times the lower bound. However, in this paper we restrict  ourselves to our main purpose of showing that Meyniel's conjecture holds \aas\ for random graphs.

Secondly, we concentrate on the cop number here, but one can also use  our winning strategy for cops to estimate the capture time, that is, the number of steps the game lasts. However, let us note that our general purpose result gives relatively weak bound on the capture time, since we did not want to introduce any additional assumptions on the graph that are not necessary for the result on the cop number. In particular, we assume only that the diameter of the giant component is at most $n$, whereas for the random graphs it would be $O(\log n/\log \log n)$.

Thirdly, note that our goal was to show that random graphs a.a.s.\ satisfy the following version of Meyniel's conjecture:  for all graphs $G$, either $G$ is  disconnected or $c(G)=O(\sqrt n)$. From that perspective, it is enough to restrict to random graphs with $d > (1-\eps) \log n$ for some $\eps >0$, since sparser graphs are \aas\ disconnected. Our results show the cop number is $O(\sqrt n)$ for  $d > (1/2+\eps) \log n$ (even though the random graph can be disconnected). This is a natural choice, since for example if $d < (1/2-\eps) \log n$ there will be too many vertices of degree zero. However, this brings up the natural question of what happens if  the robber is restricted to playing on the giant component for $d < (1/2+o(1)) \log n$. It would be very interesting to prove Meyniel's conjecture   for the giant component of the random graph in this sparse case.   We believe that some fairly serious adaptations of our argument will let $d$ be pushed significantly below $(1/2) \log n$, but there are several problems revolving around the badly behaved nature of the rate of expansion of neighbourhoods that would make it difficult to reach down as far as constant $d$. 
\medskip

\noindent {\bf Acknowledgment\ } The authors would like to thank the anonymous referee for a careful reading of the paper.

\end{document}